\documentclass[twocolumn,amsthm,secthm]{autart}
\usepackage{mathtools} 
\usepackage{amssymb}
\usepackage{amsbsy}

\usepackage[dvipdfm]{graphicx}
\usepackage{subfigure}
\usepackage{bmpsize}

\theoremstyle{plain}
\newtheorem{theorem}{Theorem}[section]
\newtheorem{lemma}[theorem]{Lemma}
\newtheorem{proposition}[theorem]{Proposition}
\newtheorem{remark}[theorem]{Remark}
\newtheorem{graph}{Graph}
\theoremstyle{definition}
\newtheorem{definition}[theorem]{Definition}
\newtheorem{assumption}[theorem]{Assumption}
\theoremstyle{remark}

\usepackage{color}

\newcommand{\adots}{\mathinner{\mkern2mu\raisebox{0.1em}{.}\mkern2mu\raisebox{0.4em}{.}\mkern2mu\raisebox{0.7em}{.}\mkern1mu}} 
\newcommand{\parallelsum}{\mathbin{\!/\mkern-5mu/\!}} 

\begin{document}

\begin{frontmatter}

\title{Intrinsic Tetrahedron Formation of Reduced Attitude\thanksref{footnoteinfo}}

\thanks[footnoteinfo]{This paper was not presented at any IFAC
meeting.
}

\author[kth,hit]{Silun Zhang}\ead{silunz@kth.se},
\author[kth]{Wenjun Song}\ead{jsongwen@163.com},
\author[hit]{Fenghua He}\ead{hefenghua@hit.edu.cn},
\author[amss]{Yiguang Hong}\ead{yghong@iss.ac.cn},
\author[kth]{Xiaoming Hu}\ead{hu@kth.se}

\address[kth]{Optimization and Systems Theory, Department of Mathematics, KTH Royal Institute of Technology, 100 44 Stockholm, Sweden}
\address[hit]{Control and Simulation Center, Harbin Institute of Technology, 150001 Harbin, China}
\address[amss]{Key Laboratory of Systems and Control, Academy of Mathematics and Systems Science, Chinese Academy of Sciences, 100190 Beijing, China}


\begin{abstract}                          
In this paper, formation control for reduced attitude is studied, in which both stationary and rotating regular tetrahedron formation can be achieved and are asymptotically stable under a large family of gain functions in the control. Moreover, by further restriction on the control gain, almost global stability of the stationary formation is obtained. In addition, the control proposed is an intrinsic protocol that only uses relative information and does not need to contain any information of the desired formation beforehand. The constructed formation pattern is totally attributed to the geometric properties of the space and the designed inter-agent connection topology. Besides, a novel coordinates transformation is proposed to represent the relative reduced attitudes in $\mathcal{S}^2$, which is shown to be an efficient approach to reduced attitude formation problems.
\end{abstract}

\end{frontmatter}


\section{Introduction}
Formation control problem has attracted considerable attention in the last decades. This trend is not only inspired by a significant amount of similar phenomenon presented in bio-communities, but also motivated by the theoretical challenges posed and its practical potential in various applications, such as formation flying \cite{Ren03ACC,Scharf04ACC}, target encircling \cite{WRen_circlement}, terrestrial or oceanographic exploration \cite{Hu01TAC}, and collaborative surveillance \cite{Tron09CDC,Wang13PRL}. Among such problems studied, attitude formation is an important one.


Rigid-body attitude control has been widely studied for a long history \cite{Morin95,Wen91}, and utilized in many engineering applications, such as the control of atmospheric aircrafts, earth-synchronous satellites, and robotic manipulators. In general, the attitude of an unconstrained rigid-body has three degrees of freedom, which can be represented globally and uniquely by a rotation matrix evolving in Lie group $\mathcal{SO}(3)$. However, in some scenarios, not all three degrees of freedom are relevant to the problem. For example, in the control of  field-of-view for cameras, orientation of the solar panel or antenna for satellites, and thrust vector for quad-rotor aircrafts, we only concern the pointing direction of a body-fixed axis, and any rotations about this axis are then indifferent. Motivated by these applications, reduced attitude control problems arise (see \cite{bullo1995control,chaturvedi2011rigid} and references therein), in which reduced attitudes evolves in  the 2-sphere $\mathcal{S}^2$.

As the configuration space of (reduced) attitudes is a compact manifold, some topological obstructions emerge in attitude control. In order to refrain from working directly on such a manifold, various attitude parameterizations are applied in attitude control studies, such as  Euler angles \cite{ren2007distributed,shuster1993survey}, quaternions \cite{lawton2002synchronized,ren2006distributed}, Rodrigues parameters \cite{zou2012attitude}. However, unfortunately, no parameterizations \cite{shuster1993survey} can represent the attitude space both globally and uniquely.
%
%
In addition to this representation barrier, as both product manifolds  $\mathcal{SO}(3)^n$ and $(\mathcal{S}^2)^n$ are compact smooth manifolds without boundary \cite{Bhat00scl}, any time-invariant Lipschitz continuous feedback control inherently yield multiple disjoint equilibria manifolds. It leads to that none of them can achieve globally asymptotical stability. Thus, almost global stabilization becomes the best possible result for attitude control problems.

Following many significant results on reduced attitude control for single rigid-body \cite{bullo1995control,chaturvedi2011rigid}, recently many results have been obtained on cooperative attitude synchronization of multiple rigid bodies. As studies for single rigid-body, most of coordination control for multiple attitudes are based on some parameterization \cite{lawton2002synchronized,paley2009stabilization,ren2010distributed,zou2016distributed}. In order to avoid topological singularities in some certain points, \cite{sarlette2009autonomous} studies consensus control directly in attitude space, however the control proposed only provides local stability of the consensus manifold. Then \cite{tron2012intrinsic} solves this problem, and achieves consensus with almost global convergence. Compared to consensus, formation control is more general and also more involved. \cite{WJ15ac} investigates formation control for reduced attitudes with ring inter-agent graph, and shows that for different parity of the number of agents, antipodal and cyclic formation are almost globally stable respectively.

As a three-dimensional configuration, tetrahedron formations have attracted a considerable interest in numerous applications and many ongoing projects \cite{curtis1999magnetospheric,escoubet1997cluster,hughes2008formation}. One of the reasons is that a tetrahedron formation consisting of four sensors is the minimum needed to resolve a physical field gradients in space \cite{Daly1998Analysis}, and especially a regular tetrahedron provides the maximal observational effectiveness \cite{roscoe2013satellite}. In the Magnetospheric Multiscale (MMS) mission launched by NASA \cite{curtis1999magnetospheric,hughes2008formation}, in order to build a three-dimensional model of the electric and magnetic fields of the Earth's magnetosphere, four satellites need to construct a regular tetrahedron formation near the apogee of a reference orbit. Another example using the tetrahedron configuration is the Cluster mission developed by  European Space Agency \cite{escoubet1997cluster}.

In this paper, a continuous control law is proposed for reduced attitudes problems, by which both stationary and rotating regular tetrahedron formations can achieve asymptotical stability under a quite large family of gain functions in the control. With a further restriction on the control gain, almost global stability of the stationary formation is also obtained. To this end, we introduce a novel coordinates transformation that represents the relative reduced attitudes between the agents. Although involved, it is shown to be an efficient approach to reduced attitude formation problems. Moreover, in contrast to most existing methodologies on formation control \cite{lawton2002synchronized,Mou15IFAC,ren2010distributed}, the proposed method does not need to have the formation errors in the protocol. Whereas, the desired formation is constructed based on the geometric properties of the manifold $\mathcal{S}^2$ and the designed connection topology. We referred to this type of formation control as intrinsic formation control. Another virtue of the control proposed is that only relative attitude measurement is required. Although only reduced attitudes are considered in the paper, similar design techniques can also be applied to the formation control of point mass systems\cite{Zhang16CCC}.

The rest of the paper is organized as follows: in Section~\ref{sec:pre}, the necessary preliminaries and notions are introduced.  Section~\ref{sec:pformu} proposes the intrinsic formation problem addressed in the paper. In Section~\ref{sec:SF}, the stationary regular tetrahedron formation is rendered almost globally asymptotically stable. Following that, an illustrative example is provided in Section~\ref{sec:simu}, and the conclusions are given in Section~\ref{sec:conclusion}.

\section{Notation and Preliminary}\label{sec:pre}
In this paper, we model the inter-agent connectivity as a graph $\mathcal{G}=(\mathcal{V},\mathcal{E})$, where $\mathcal{V}=\{1,\ldots,n\}$  is the set of nodes, and $\mathcal{E}\subset \mathcal{V} \times \mathcal{V}$ is the edge set. A graph $\mathcal{G}$ is referred to as an undirected graph if $(j,i) \in \mathcal{E}$, for every $(i,j) \in \mathcal{E}$, otherwise $\mathcal{G}$ is directed. We also define the neighbor set of node $i$ as $\mathcal{N}_i= \{ j:(j,i) \in \mathcal{E}\}$, and we say $j$ is a neighbor of $i$, if $j \in \mathcal{N}_i$.
\subsection{Attitude and reduced attitude}

 In the three-dimensional Euclidean space, the (full) attitude of a rigid body $i$ can be determined by a length-preserving linear transformation between coordinates frame $\mathcal{F}_i$ and $\mathcal{F}_w$, where $\mathcal{F}_i$ is the instantaneous body frame of agent $i$ and $\mathcal{F}_w$ presents an inertial reference frame. This transformation is specified by a \textit{rotation matrix} $R_i\in \mathbb{R}^{3\times3}$, and its columns are, respectively, coordinates of three orthogonal unit basis of frame $\mathcal{F}_i$ resolved to frame $\mathcal{F}_w$. Furthermore, all rotation matrices together form the special orthogonal group $\mathcal{SO}(3)=\{R\in\mathbb{R}^{3\times3}: R^TR=I, det(R)=1\}$,  under the operation of matrix multiplication.


In reduced attitude applications, we only consider the pointing direction of a body-fixed axis, and ignore the rotations about this axis. Let $\boldsymbol{b}_i\in \mathcal{S}^2$ denote the coordinates of $i$'s pointing axis relative to the body frame $\mathcal{F}_i$, where $\mathcal{S}^2=\{x\in\mathbb{R}^{3}: \|x\|=1\}$ is the $2$-sphere. Then $i$'s pointing axis coordinates resolved in frame $\mathcal{F}_w$ is $\Gamma_i=R_i\boldsymbol{b}_i$. Since $R_i \in \mathcal{SO}(3)$, we still have $\Gamma_i \in \mathcal{S}^2$. This vector $\Gamma_i$ is referred to as the \textit{reduced attitude} of rigid body $i$, on account of the neglect of the rotations about one axis in pointing applications.

 The kinematics of the reduced attitude $\Gamma_i$ is governed by \cite{WJ15ac}
 \begin{equation}\label{eq:pre:DGamma}
        \dot{\Gamma}_i = \widehat{\omega}_i \Gamma_i ,
\end{equation}
where $\omega_i \in \mathbb{R}^{3}$ is $i$'s angular velocity relative to the inertial frame $\mathcal{F}_w$, and hat operator $\widehat{(\cdot)}$ satisfies, for any $x=[x_1,x_2,x_3]^T\in \mathbb{R}^{3}$,
\begin{equation}\label{eq:pre:hat}
 \widehat{x}:=
 \begin{pmatrix}
   0 &-x_3 &x_2\\
   x_3& 0 &-x_1\\
   -x_2& x_1 & 0
 \end{pmatrix} .
\end{equation}
%

For any two points $\Gamma_i, \Gamma_j \in \mathcal{S}^2$, we define $\theta_{ij} \in [0, \pi]$ and $k_{ij} \in \mathcal{S}^2$ as
$$ \theta_{ij}=\arccos(\Gamma_i^T \Gamma_j), \;\; k_{ij}=\frac{\widehat{\Gamma_i} \Gamma_j}{\sin(\theta_{ij})}.$$
In the definition of $k_{ij}$, we stipulate $k_{ij}$ to be any unit vector orthogonal to $\Gamma_i$, when $\theta_{ij}=0$ or $\pi$.

%

Note that $\theta_{ij}$ is also the geodesic distance between $\Gamma_i$ and $\Gamma_j$ in $\mathcal{S}^2$, and $\Gamma_j= exp(\theta_{ij} \widehat{k}_{ij})\Gamma_i$. The relationship among three reduced attitudes follows from the spherical cosine formula \cite{Todhunter_Identity}:

\begin{lemma}\label{Thm:PR:cosine}
  For any three reduced attitudes $\Gamma_i, \Gamma_j$, $\Gamma_k \in \mathcal{S}^2$, the following relationship will always hold:
$$\cos(\theta_{ij})=\cos(\theta_{ik})\cos(\theta_{jk})+\sin(\theta_{ik})\sin(\theta_{jk})k_{ik}^Tk_{jk}.$$
\end{lemma}

In this paper, we will also use a frequently mentioned parametrization of $\Gamma_i$ based on RPY angles system \cite{Sciavicco_Book},
\begin{gather}\label{eq:pre:RPY}
  \Gamma_i=
    \begin{pmatrix}
      cos(\psi_i)cos(\phi_i)\\
      sin(\psi_i)cos(\phi_i)\\
      sin(\phi_i)
    \end{pmatrix}.
\end{gather}
where $\psi_i \in [-\pi,\pi)$, $\phi_i \in [-\pi/2,\pi/2]$.

\subsection{Several Results on Matrices}
In this part, we list several results used in the paper on matrices.

\begin{definition}
  $A,B \in \mathbb{R}^{n \times n}$, we say that $A$ and $B$ can be \emph{simultaneously diagonalized} if there exists a nonsingular matrix $P$ such that
   $$P^{-1}AP=D_1,\; P^{-1}BP=D_2,$$
   where $D_1$ and $D_2$ are both diagonal matrices.
\end{definition}

The following theorem gives the sufficient and necessary condition regarding simultaneous diagonalization.
\begin{lemma}  \cite{horn2012matrix} \label{Thm:pre:simuDiagonalize} 
  If $A,B \in \mathbb{R}^{n \times n}$ and assume that $A$, $B$ are both diagonalizable, then $A$ and $B$ are simultaneously diagonalizable if and only if $A$, $B$ are commutative i.e. $AB=BA$.
\end{lemma}
%

Next, we revisit a special class of matrices, \emph{persymmetric matrix}. A matrix is referred to persymmetric matrix when it is symmetric about its counterdiagonal. A common persymmetric matrix is exchange matrix, denoted by $J$, in which the $1$ elements reside on its counterdiagonal and all other elements are $0$. We can verify that
$$J^TJ=I,\;\;\; J=\begin{pmatrix}
      & \, & \;1 \\
      & \adots &\,    \\
    1\; & \,  &\,
  \end{pmatrix} ,$$
where $I$ is the identity matrix in a suitable size.

A real matrix $A$ is persymmetric, if and only if $A=JA^TJ$. So we have the following lemma.
\begin{lemma}\label{Thm:PR:AJCommutative}
  Let $A \in \mathbb{R}^{n \times n}$ be a real matrix. If $A$ is both symmetric and persymmetric, then $A$ and $J$ are commutative, i.e. $AJ=JA$.
\end{lemma}

The next theorem provides the spectrum of a both symmetric and persymmetric matrix in a special structure.
\begin{lemma}\label{Thm:PR:lambdaMatrix}
   Let $\Lambda_n=\{a_{ij}\}\in \mathbb{R}^{n\times n}$. If its element $a_{ij}$ satisfies
    $$a_{ij}=\begin{cases}
      0, & if\; j=n-i+1\\
      1, & \;otherwise
    \end{cases},
    $$
    then the eigenvalues of $\Lambda_n$ is
    $$\begin{cases}\lambda_1(\Lambda_n)=n-1;\\
      \lambda_i(\Lambda_n)=(-1)^{i};\; i \in \{2,\cdots,n\}.
    \end{cases} $$
\end{lemma}
This theorem follows from the fact that $\Lambda_n=E_n-J$, where $E_n$ is the matrix whose all elements are $1$, meanwhile $E_n$ and $J$ can be simultaneously diagonalized.
The detailed  proof of this theorem can be found in the appendix.

\section{Problem Formulation}\label{sec:pformu}

In this paper, we consider only \textit{intrinsic formation control problem}, namely we try to design a distributed control law not containing any information about the desired formation. The construction of desired formations is totally attributable to the geometric properties of the space and the connection topology between the agents. The control law solving this kind of problems is referred to as \textit{intrinsic control law}.

In this paper, we focus on a specific solid formation, regular tetrahedron, in $\mathcal{S}^2$, in which the geodesic distance between every two reduced attitudes are identical, i.e. $d_{\mathcal{S}^2}(\Gamma_i,\Gamma_j)=\theta^*$, for any $i \neq j$, and $i,j \in \{1,2,3,4\}$. From the elementary geometry, we have $\theta^*=\arccos(-\frac{1}{3})$. We denote the set of regular tetrahedron formations by $\Omega_T$, where
\begin{equation}\label{eq:PF:OmegaT}
  \Omega_T=\left\{ \mathbf{\Gamma} \in (\mathcal{S}^2)^4 : \Gamma^T_i\Gamma_j = -\frac{1}{3}, \forall i \neq j \right\}.
\end{equation}
%

In our previous work \cite{WJ15ac}, we propose a control only containing the relative attitude $\{\widehat{\Gamma}_i\Gamma_j:j\in \mathcal{N}_i\}$ to achieve antipodal and cyclic formation for the ring graph topology. Here, we employ a similar but modified control for tetrahedron formation,

\begin{equation}\label{eq:PF:Control}
  \omega_i=-\sum_{j\in \mathcal{N}_i} f(\theta_{ij})\widehat{\Gamma}_i\Gamma_j, \; i\in \mathcal{V},
\end{equation}
where $\mathcal{V}=\{1,2,3,4\}$, $f: \mathbb{R} \longrightarrow \mathbb{R}$ is a real function.

 We note that, in this protocol, $\widehat{\Gamma}_i\Gamma_j$ is a relative reduced attitude information. Its coordinates resolved in the body frame of agent $i$ is $R_i^T \widehat{\Gamma}_i\Gamma_j = \mathbf{b}_i \times (R_i^T R_j\mathbf{b}_j)$, and $R_i^T R_j\mathbf{b}_j$ is just the body-fixed pointing axis of agent $j$ observed from agent $i$. Similarly, $\cos(\theta_{ij})=\Gamma_i^T\Gamma_j=\mathbf{b}_i^T(R_i^T R_j\mathbf{b}_j)$ is also relative information which can be measured from the body-frame of agent $i$. This is summarized as follows,
\begin{remark}\label{Rmk:PF:Control}
   The control law (\ref{eq:PF:Control}) only contains relative information between reduced attitudes.
\end{remark}

Substituting (\ref{eq:PF:Control}) into (\ref{eq:pre:DGamma}), the closed-loop system is obtained as

\begin{equation}\label{eq:PF:ClosedSystem}
  \dot{\Gamma}_i=\widehat{\Gamma}_i \sum_{j\in \mathcal{N}_i} f(\theta_{ij})\widehat{\Gamma}_i\Gamma_j, \; i\in \mathcal{V}.
\end{equation}

Using Rodrigues formula, one can show that the closed-loop system (\ref{eq:PF:ClosedSystem}) is invariant under any rotation operations.
\begin{lemma}\label{Thm:PF:rotationalInvariant}
For any rotation transformation about a unit axis $u \in \mathcal{S}^2$ through an angle $\theta \in [0,\pi]$, the system  (\ref{eq:PF:ClosedSystem}) is invariant, i.e. if
$$\widetilde{\Gamma}_i = exp(\theta \widehat{u})\Gamma_i, \; i \in \mathcal{V}$$
then the closed-loop system in terms of $ \widetilde{\Gamma}_i$ is
$$ \dot{\widetilde{\Gamma}}_i=\widehat{\widetilde{\Gamma}}_i \sum_{j\in \mathcal{N}_i} f(\theta_{ij})\widehat{\widetilde{\Gamma}}_i\widetilde{\Gamma}_j, \; i\in \mathcal{V}.$$
\end{lemma}

Using the parametrization (\ref{eq:pre:RPY}), after some rather tedious computation, the closed-loop system (\ref{eq:PF:ClosedSystem}) under this coordinates can be written as
\begin{eqnarray}
&\cos(\phi_i)\dot{\psi}_i=\! \sum\limits_{j\in \mathcal{N}_i}{f(\theta_{ij})\sin(\psi_i\!-\!\psi_j) \cos(\phi_j)}, \label{eq:PF:ClosedSysRPY1}\\
&\dot{\phi}_i \!=\!\!\!\sum\limits_{j\in \mathcal{N}_i}{\!f(\theta_{ij})\!\left[ \sin(\phi_i\!) \cos(\psi_j\!-\!\psi_i\!) \!-\! \cos(\phi_i\!) \sin(\phi_j\!) \right]} \label{eq:PF:ClosedSysRPY2},
\end{eqnarray}
where $i \in \mathcal{V}$ and the geodesic distance $\theta_{ij}$ can be computed in terms of $\{\phi_i,\psi_i\}_{i \in \mathcal{V}}$ as following
\begin{equation}\label{eq:PF:ClosedSysRPYRef}
\begin{aligned}
 \theta_{ij}=arccos \big(\cos(\phi_i) \cos(\phi_j)&\cos(\psi_i -\psi_j)\\
 &+ \sin(\phi_i) \sin(\phi_j)  \big).
\end{aligned}
\end{equation}


\section{Stationary Regular Tetrahedron Formation}\label{sec:SF}
In this section, we will investigate the almost global stability of a stationary regular tetrahedron formation under the proposed intrinsic control. Since the intrinsic formation problem depends on both the spacial properties and the inter-agent topology, we need to design the topology of inter-agent connection as follows:
\begin{graph}\label{As:SF:completegraph}
  The inter-agent graph $\mathcal{G}$ is a completed graph, i.e. $(i,j) \in \mathcal{E}$, for any $i,j \in \mathcal{V}$ and $i\neq j$.
\end{graph}

\subsection {Stability of Equilibria Set}
Next, we show that under \emph{graph} \ref{As:SF:completegraph}, every trajectory of closed-loop system (\ref{eq:PF:ClosedSystem}) will asymptomatically converge to some equilibrium.
\begin{theorem}\label{Thm:SF:stableEquiSet}
  Under Graph \ref{As:SF:completegraph}, if the function $f(\cdot)$ in control (\ref{eq:PF:Control}) is nonnegative and bounded almost everywhere on $[0,\pi]$ ,
then for closed-loop system (\ref{eq:PF:ClosedSystem}), the equilibria set  $\Omega$ is globally asymptomatically stable, where
\begin{equation}\label{eq:PF:Omega}
\Omega = \Big\{  \mathbf{\Gamma} \in (\mathcal{S}^2)^4:  \sum_{j\in \mathcal{N}_i} f(\theta_{ij})\widehat{\Gamma}_i\Gamma_j=0, i\in \mathcal{V} \Big\}.
\end{equation}
\end{theorem}

\begin{proof}
  Define a candidate Lyapunov function $V(\mathbf{\Gamma})$ as
$$V(\mathbf{\Gamma})=\sum\limits_{(i,j) \in \mathcal{E}} \int_{{\theta _{ij}}}^\pi  {f(x)\sin (x)dx} ,$$
where $\theta_{ij}=\arccos(\Gamma_i^T \Gamma_j)$. Since $\theta _{ij} \in [0,\pi]$, $V(\mathbf{\Gamma}) \geq 0$, for any $\mathbf{\Gamma} \in (\mathcal{S}^2)^4 $. Due to the fact $\widehat{\Gamma}_i^T=-\widehat{\Gamma}_i$, the time derivative of $V(\mathbf{\Gamma})$ along the trajectory of the closed-loop system (\ref{eq:PF:ClosedSystem}) satisfies
\begin{equation*}
    \begin{split}
      \dot{V}(\mathbf{\Gamma})&=\sum\limits_{(i,j) \in \mathcal{E}} -f(\theta_{ij})\sin(\theta_{ij})(\theta_{ij})'\\
      &= \sum\limits_{(i,j) \in \mathcal{E}} f(\theta_{ij})\left(\Gamma_i^T \Gamma_j \right)'\\
        &=-\sum\limits_{(i,j) \in \mathcal{E}}  f(\theta_{ij})\Bigg((\widehat{\Gamma}_j \Gamma_i)^T  \sum_{k\in \mathcal{N}_j} f(\theta_{jk})\widehat{\Gamma}_j\Gamma_k  \\
        &\;\;\;\;\;\; \;\;\;\;\;\;\;\; \;\;\;\;\; \;\;\;\;\;\;\;\;  +(\widehat{\Gamma}_i \Gamma_j)^T   \sum_{k\in \mathcal{N}_i}  f(\theta_{ik})\widehat{\Gamma}_i\Gamma_k   \Bigg)\\
        &=-2\sum\limits_{i=1}^4 \Bigg( \sum_{j\in \mathcal{N}_i} f(\theta_{ij})\widehat{\Gamma}_i\Gamma_j  \Bigg)^2
    \end{split}
\end{equation*}
%

We note that the equilibria set $\Omega$ is the largest invariant set contained in itself. Since $\dot{V}(\mathbf{\Gamma}) \leq 0$ and $\mathcal{S}^2$ is a compact space, by LaSalle's invariance principle \cite{KhalilBook}, the trajectory of the system (\ref{eq:PF:ClosedSystem}) will globally asymptomatically approach the set $\Omega$.
\end{proof}

Revisiting our desired formations, we have the following relationship between $\Omega$ and $\Omega_T$:
\begin{remark}\label{Rmk:SF:OmegaTinEqm}
  Under the graph \ref{As:SF:completegraph}, regardless of the choice of function $f(\cdot)$ in control (\ref{eq:PF:Control}), the regular tetrahedron formations are always equilibria of closed-loop system (\ref{eq:PF:ClosedSystem}), i.e. $\Omega_T \subseteq \Omega$.
\end{remark}

Since $(\mathcal{S}^2)^n$ is a compact space without boundary, any continuous time-invariant feedback control always yields multiple disjoint equilibria manifolds \cite{Bhat00scl,WJ15ac}. This leads to that there exist other equilibria manifolds in $\Omega$ except regular tetrahedron formations $\Omega_T$. Hence though Theorem \ref{Thm:SF:stableEquiSet} reveals that the whole equilibria set is globally asymptotically stable, the stability of desired formations $\Omega_T$ is not settled yet.

\subsection {Stability of Regular Tetrahedron Formations}
This subsection will focus on the local stability of the desired formation. We show first the nonlinear system (\ref{eq:PF:ClosedSystem}) possesses a 3-dimensional center manifold. Furthermore, by introducing a novel coordinates transformation, the analytic description of this center manifold is obtained. Based on this, it is shown that this center manifold is exactly $\Omega_T$, which guarantees the local stability of the desired formation.

\begin{assumption} \label{As:SF:f}
  The gain function $f(\cdot)$ in control (\ref{eq:PF:Control}) is bounded and satisfies
$$f(\cdot)> 0, \; \dot{f}(\cdot)< 0 \;\; on \;(0,\pi) .$$
\end{assumption}

Then we first investigate the stability of a specific reduced attitude $\mathbf{\Gamma}^*\in \Omega_T$, whose parametrization under the coordinates (\ref{eq:pre:RPY}), denoted by $\{(\psi^*_i,\phi^*_i)\}_{i \in \mathcal{V}}$ , satisfies:

\begin{equation}\label{eq:SF:proofPoint}
\begin{array}{cc}
  \begin{cases}
  \psi^*_i=\frac{\pi}{4} + (i-1)\frac{\pi}{2} \\
  \phi^*_i=(-1)^{(i-1)}arcsin(\frac{\sqrt{3}}{3})
\end{cases} & i \in \mathcal{V}
\end{array}.
\end{equation}

Let $\overline{\psi}_i=\psi_i-\psi^*_i$, $\overline{\phi}_i=\phi_i-\phi^*_i$. The closed-loop system (\ref{eq:PF:ClosedSysRPY1})-(\ref{eq:PF:ClosedSysRPY2}) can be expanded near $\mathbf{\Gamma}^*$ as
\begin{equation}\label{eq:SF:proofLinear}
\dot{\zeta}=A _\zeta(\mathbf{\Gamma}^*) \zeta+ g(\zeta),
\end{equation}
where
$\zeta=[\overline{\phi}_1,\overline{\psi}_1,\overline{\phi}_2,\overline{\psi}_2,\overline{\phi}_3,\overline{\psi}_3,\overline{\phi}_4,\overline{\psi}_4]^T$, and $g({0})={0}$, $\frac{\partial g}{\partial \zeta}({0})=0$. The fact $\mathbf{\Gamma}^*\in \Omega_T$ provides that $\theta^*_{ij}=\arccos(-\frac{1}{3})$, for any $i \neq j$. So we can separate $f(\cdot)$,$\dot{f}(\cdot)$ in the matrix $A _\zeta(\mathbf{\Gamma}^*)$, and $A _\zeta(\mathbf{\Gamma}^*)$ will be rewritten as the following form:
$$ A _\zeta(\mathbf{\Gamma}^*)=A_{\zeta}^1\cdot f(\theta^*)+ A_{\zeta}^2\cdot \dot{f}(\theta^*), $$
where $A_{\zeta}^1, A_{\zeta}^2 \in \mathbb{R}^{8 \times 8}$, $\theta^*=\arccos(-\frac{1}{3})$.

Now we can explain why we elaborately select a specific $\mathbf{\Gamma}^*$ in (\ref{eq:SF:proofPoint}) . By this choice, we make the matrices $A_{\zeta}^1$ and $A_{\zeta}^2$ commutative, i.e. $A_{\zeta}^1 A_{\zeta}^2=A_{\zeta}^2 A_{\zeta}^1$. By \emph{Lemma} \ref{Thm:pre:simuDiagonalize},there exists an invertible matrix $P$ such that
$$ A _\zeta(\mathbf{\Gamma}^*)=P^{-1} \left( D^1\cdot f(\theta^*)+ D^2\cdot \dot{f}(\theta^*)\right)P, $$
where $D^1,D^2 \in \mathbb{R}^{8 \times 8}$ and they are both diagonal matrices. So the elements on the diagonal of matrix $D^1\cdot f(\theta^*)+ D^2\cdot \dot{f}(\theta^*)$ are the spectrum of $A _\zeta(\mathbf{\Gamma}^*)$, which can be computed as
\begin{equation}\label{eq:SF:proofEigen}
\begin{split}
  &\lambda_{1,2}=2\sqrt2\dot{f}(\theta^*),\\
  &\lambda_{3,4,5}=\frac{4\sqrt2}{3}\dot{f}(\theta^*)-\frac{8}{3}f(\theta^*),\\
  &\lambda_{6,7,8}=0.
\end{split}
\end{equation}
By the center manifold theory, it is implied that under Assumption \ref{As:SF:f}, the closed-loop system (\ref{eq:PF:ClosedSystem}) possesses a 3-dimensional center manifold $M_c(\mathbf{\Gamma}^*)$ near $\mathbf{\Gamma}^*$.

Furthermore for any formation $\mathbf{\Gamma} \in \Omega_T$, we can transform it to $\mathbf{\Gamma}^*$ by a rotational transformation, according to \emph{Lemma} \ref{Thm:PF:rotationalInvariant}, then we obtain the following lemma.

\begin{lemma}\label{Thm:SF:ExistCM}
Under Graph \ref{As:SF:completegraph} and Assumption \ref{As:SF:f}, for closed-loop system (\ref{eq:PF:ClosedSystem}), there exists a 3-dimensional center manifold $M_c(\mathbf{\Gamma})$ at any formation $\mathbf{\Gamma}\in \Omega_T$.
\end{lemma}

\emph{Lemma}~\ref{Thm:SF:ExistCM} provides the existence of a 3-dimensional center manifold $M_c(\mathbf{\Gamma})$. However the center manifold of a nonlinear system, in general, cannot be expressed analytically. In the following part, we introduce a new coordinate transformation, by which the analytical description of $M_c(\mathbf{\Gamma})$ can be derived.

\begin{figure}[t]
  \centering
  \includegraphics[width=0.5\textwidth]{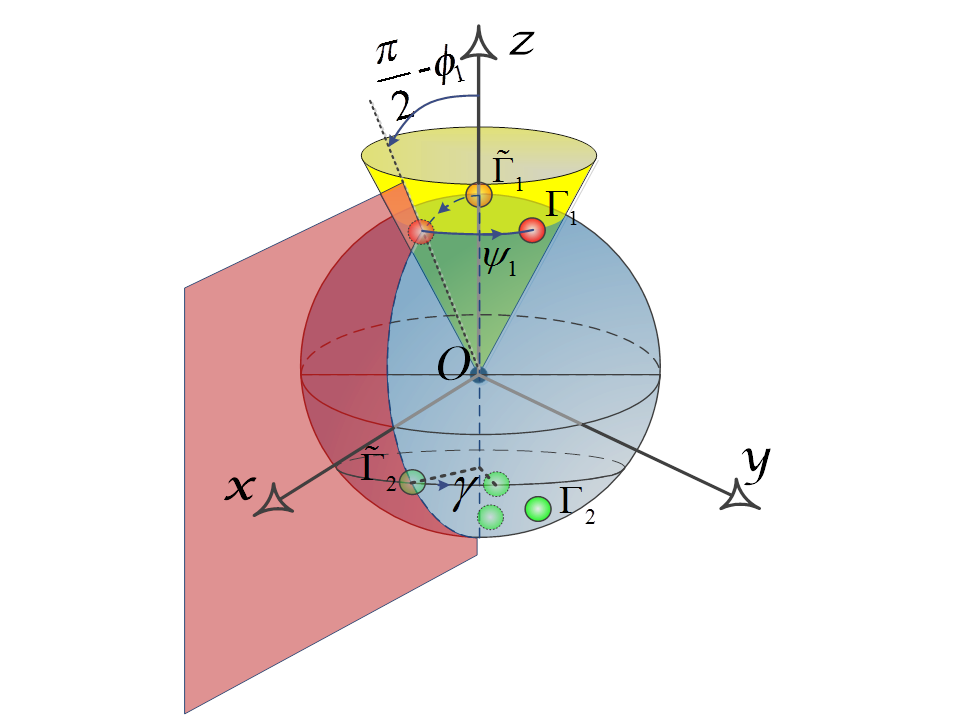}\\
  \caption{Depiction of coordinates for absolute components}\label{Fig:SF:Coordinate}
\end{figure}

We set a new coordinates system consisting of the relative attitude between every two agents $i$,$j$ and the absolute attitude of the whole formation. This transformation is denoted by $\xi=[\xi_s^T,\xi_c^T]^T=\Phi(\mathbf{\Gamma})$, for every $\mathbf{\Gamma}\in (\mathcal{S}^2)^4$.

In coordinates $\xi$, the component $\xi_s=[\cos(\theta_{12}),\cos(\theta_{13}),$\\$\cos(\theta_{14}),\cos(\theta_{23}),\cos(\theta_{24}),\cos(\theta_{34})]^T$ represents the relative attitude. The absolute attitude component $\xi_c=[\phi_1,\psi_1,\gamma]^T$ describes the attitude of $\Gamma_1, \Gamma_2$, or equivalently the whole formation, relative to the inertial frame $\emph{O-XYZ}$. More specifically, the first two elements of $\xi_c$ are the RPY coordinates of $\Gamma_1$ satisfying the parametrization (\ref{eq:pre:RPY}). The third element in $\xi_c$ can be determined from the following rotation operations, which is illustratively depicted in Fig.\ref{Fig:SF:Coordinate}. Let $\widetilde{\Gamma}_1=[0,0,1]^T$ overlapping with the north pole, and $\widetilde{\Gamma}_2=[\sin(\theta_{12}),0,\cos(\theta_{12})]^T$ on the plane $\emph{XOZ}$. Then $\widetilde{\Gamma}_1$, $\widetilde{\Gamma}_2$ can coincide with $\Gamma_1$ and $\Gamma_2$ respectively, by a series of rotations consisting of rotating about the \emph{Z}-axis trough an angle $\gamma$, then the \emph{Y}-axis through $\frac{\pi}{2}-\phi_1$, and finally the \emph{Z}-axis through $\psi_1$. Hence, we can express the transformation equation corresponding to these rotation operation as $\Gamma_j= R_3(\psi_1)R_2(\phi_1)R_1(\gamma)\widetilde{\Gamma}_j$, where $j\in \{1,2\}$ and
\begin{eqnarray}\label{eq:SF:rotMatrix}
         &&R_1(\gamma)\!=\!
         \begin{pmatrix}
           \cos\!\gamma & -\!\sin\!\gamma &0\\
           \sin\!\gamma &\cos\!\gamma & 0\\
           0& 0 & 1
         \end{pmatrix}\!,
         R_2(\phi_1)\!=\!
         \begin{pmatrix}\!
           \sin\!\phi_1 & \!0 &\cos\!\phi_1\!\\
           0 & \!1 & 0\\
           \!-\!\cos\!\phi_1 & \!0 & \sin\!\phi_1\!
         \end{pmatrix},\nonumber \\
        &&R_3(\psi_1)\!=\!
         \begin{pmatrix}
           \cos\!\psi_1&-\!\sin\!\psi_1&0\\
           \sin\!\psi_1&\cos\!\psi_1&0\\
           0& 0 & 1
         \end{pmatrix}.
\end{eqnarray}
From this rotation transformation, the third element of $\xi_c$, angle $\gamma$, can be determined.

Now, we can explicitly work out this coordinate transformation for $\mathbf{\Gamma}\in (\mathcal{S}^2)^4$:
\begin{gather}\label{eq:SF:newCoorTrans}
\xi=\Phi(\mathbf{\Gamma})=[\xi_s^T,\xi_c^T]^T,
\end{gather}
 and every element of $\xi_s$ and $\xi_c$ can be completely computed as follows:
\begin{equation}\label{eq:SF:newCoorTransPara}
\begin{split}
  \theta_{ij}&=\arccos(\Gamma_i^T\Gamma_j), \;\; i,j \in \mathcal{V}\\
  \phi_k &=  \arcsin([\Gamma_{k}]_3), \\
  \psi_k &= \mathrm{atan2}([\Gamma_{k}]_2,[\Gamma_{k}]_1), \;\; k\in \{1,2\}\\
  \gamma\; &=  \mathrm{atan2}( \gamma_y,\gamma_x ),\\
  \gamma_y&=\cos(\phi_2)\sin(\psi_2-\psi_1),\\
  \gamma_x&=\sin(\phi_1\!)\cos(\phi_2\!)\cos(\psi_2\!-\!\psi_1\!)\!-\!\cos(\phi_1\!)\sin(\phi_2\!),
\end{split}
\end{equation}
where $[\cdot]_p$ is the $p$-th element of a 3-dimensional vector, and $\mathit{atan2}(\cdot)$ is the arctangent function with two arguments \cite{ZXLi}. For integrity of the definition (\ref{eq:SF:newCoorTransPara}), we supplement additional definitions for two singular cases, which are $\psi_k=0$ when $\Gamma_{k}=[1,0,0]^T$ and $\gamma=0$ when $\Gamma_{1}=\Gamma_{2}$. After this amendment, we have the following remark.

\begin{remark}
  The coordinate transformation $\xi=\Phi(\mathbf{\Gamma})$ in (\ref{eq:SF:newCoorTrans}) is diffeomorphic everywhere except on the boundary.
\end{remark}

We note that although the coordinates $\xi$ in (\ref{eq:SF:newCoorTrans}) contains nine elements, the number of independent elements is only eight. This is the consequence of the following Lemma:

\begin{lemma}
  For any $\mathbf{\Gamma}\in (\mathcal{S}^2)^4$, we have the following identity \cite{Todhunter_Identity}:
  \begin{equation}\nonumber
    \begin{split}
    &\cos^{2}\!(\theta_{12}\!)\!\!+\!\cos^{2}\!(\theta_{13}\!)\!\!+\!\cos^{2}\!(\theta_{14}\!)\!\!+\!\cos^{2}\!(\theta_{23}\!)\!\!+\!\cos^{2}\!(\theta_{24}\!)\!\!+\!\cos^{2}\!(\theta_{34}\!)\\
    &-\!\cos^{2}\!(\theta_{14}\!)\cos^{2}\!(\theta_{23}\!)\!-\!\cos^{2}\!(\theta_{13}\!)\cos^{2}\!(\theta_{24}\!)\!-\!\cos^{2}\!(\theta_{12}\!)\cos^{2}\!(\theta_{34}\!)\\
    &-2[\cos(\theta_{12}\!)\cos(\theta_{13}\!)\cos(\theta_{23}\!)\!+\!\cos(\theta_{23}\!)\cos(\theta_{24}\!)\cos(\theta_{34}\!)\\
    &\;\;\;\;\;\;+\!\cos(\theta_{13}\!)\cos(\theta_{14}\!)\cos(\theta_{34}\!)\!+\!\cos(\theta_{12}\!)\cos(\theta_{14}\!)\cos(\theta_{24}\!)]\\
    &+2[\cos(\theta_{13}\!)\cos(\theta_{14}\!)\cos(\theta_{23}\!)\cos(\theta_{24}\!)\!+\!\cos(\theta_{12}\!)\cos(\theta_{13}\!)\\
    &\;\;\;\;\;\;\times\cos(\theta_{24}\!)\cos(\theta_{34}\!)\!+\!\cos(\theta_{12}\!)\cos(\theta_{14}\!)\cos(\theta_{23}\!)\cos(\theta_{34}\!)]\\
    &\equiv 1,
    \end{split}
  \end{equation}
  where $\theta_{ij}=\arccos(\Gamma_i^T\Gamma_j)$, and $i,j \in \mathcal{V}$.
\end{lemma}
Therefore the degrees of freedom of $\xi$ are  in fact consistent with those of $(\mathcal{S}^2)^4$.

Then by transformation (\ref{eq:SF:newCoorTrans}) and \emph{Lemma} \ref{Thm:PR:cosine}, after some involved algebraic manipulation, the closed-loop dynamics (\ref{eq:PF:ClosedSystem}) becomes
\begin{subequations}\label{eq:SF:newCorDynf}
    \begin{equation}\label{eq:SF:newCorDynf1}
      \dot{\xi}_s= \hat{f}_s(\xi_s),
    \end{equation}
    \begin{equation}\label{eq:SF:newCorDynf2}
        \dot{\xi}_c= \hat{f}_c(\xi_c,\xi_s).
    \end{equation}
\end{subequations}
The expression of $\hat{f}_s(\cdot)$ in (\ref{eq:SF:newCorDynf1}) can be derived as (\ref{eq:SF:newCorDynf1Detial}), where the matrix function $h:\mathbb{R}^{6} \rightarrow \mathbb{R}^{6\times 6}$, $\widehat{F}(\xi_s)=[F(\xi_1),F(\xi_2),F(\xi_3),F(\xi_4),F(\xi_5),F(\xi_6)]^T$, $\xi_p$ is the $p$-th element of $\xi$ and $F(\cdot)$ is the composition of function $f(\cdot)$  with the inverse cosine function $arccos(\cdot)$, i.e. $F=f\circ arccos:[-1,1]\rightarrow \mathbb{R}$.
\begin{figure*}[!t]
\normalsize
\begin{align}\label{eq:SF:newCorDynf1Detial}
      \hat{f}_s(\xi_s)=h(\xi_s)\widehat{F}(\xi_s)=
       \begin{bmatrix}
          2\xi_1^2-2 \;&\; \xi_1\xi_2-\xi_4 \;&\; \xi_1\xi_3-\xi_5 \;&\; \xi_1\xi_4-\xi_2 \;&\; \xi_1\xi_5-\xi_3 \;&\; 0 \\
          \xi_1\xi_2-\xi_4 & 2\xi_2^2-2 & \xi_2\xi_3-\xi_6 & \xi_2\xi_4-\xi_1 & 0 & \xi_2\xi_6-\xi_3 \\
          \xi_1\xi_3-\xi_5 & \xi_2\xi_3-\xi_6 & 2\xi_3^2-2 & 0 & \xi_3\xi_5-\xi_1 & \xi_3\xi_6-\xi_2 \\
          \xi_1\xi_4-\xi_2 & \xi_2\xi_4-\xi_1 & 0 & 2\xi_4^2-2 & \xi_4\xi_5-\xi_6 & \xi_4\xi_6-\xi_5 \\
          \xi_1\xi_5-\xi_3& 0 & \xi_3\xi_5-\xi_1 & \xi_4\xi_5-\xi_6 & 2\xi_5^2-2 & \xi_5\xi_6-\xi_4 \\
          0 & \xi_2\xi_6-\xi_3 & \xi_3\xi_6-\xi_2 & \xi_4\xi_6-\xi_5 & \xi_5\xi_6-\xi_4 & 2\xi_6^2-2 \\
        \end{bmatrix}
        \begin{bmatrix}
            F(\xi_1) \\
            F(\xi_2) \\
            F(\xi_3) \\
            F(\xi_4) \\
            F(\xi_5) \\
            F(\xi_6)
          \end{bmatrix}.
     \end{align}
\hrulefill
\vspace*{4pt}
\end{figure*}

Although the analytic expression for $\hat{f}_c$ is cumbersome to work out, this coordinate transformation have the following advantages. Firstly, we notice that through this transformation ,  in the closed-loop system (\ref{eq:SF:newCorDynf1})-(\ref{eq:SF:newCorDynf2}) the dynamics of $\xi_s$ only depend on $\xi_s$ itself, namely the system achieves a triangular form. Another advantage is stated in the following remark.
\begin{remark}\label{Rmk:SF:maniToPoint}
  After the coordinates transformation (\ref{eq:SF:newCoorTrans}), the set of desired formations---regular tetrahedron--- is assembled into one equilibrium of the subsystem  (\ref{eq:SF:newCorDynf1}).
\end{remark}

This is because we can verify that for any $\mathbf{\Gamma}^*\in \Omega_T$, after the transformation $\xi^*=\Phi(\mathbf{\Gamma}^*)=[\xi_s^{*T},\xi_c^{*T}]^T$, the element $\xi_s^*$ satisfies
\begin{gather}\label{eq:SF:equiNewCoor}
\xi_s^*=[-\frac{1}{3},-\frac{1}{3},-\frac{1}{3},-\frac{1}{3},-\frac{1}{3},-\frac{1}{3}]^T.
\end{gather}
Furthermore, $\hat{f}_s(\xi_s^*)=0$. Then we can rewrite $\Omega_T$ in the new coordinate systems as
\begin{gather}\label{eq:SF:OmegaTNewCoor}
\Omega_T=\{(\xi_s^T,\xi_c^T)^T:\xi_s= \xi_s^* \}.
\end{gather}
Now we can obtain the analytic description of the center manifold $M_c(\mathbf{\Gamma})$.

\begin{theorem}\label{Thm:SF:CMDetile}
Under Graph \ref{As:SF:completegraph} and Assumption \ref{As:SF:f}, the following holds
\begin{enumerate}
  \item The manifold $\Omega_T$ is the center manifold of the closed-loop system (\ref{eq:PF:ClosedSystem}) at any formation $\mathbf{\Gamma}^*\in \Omega_T$.
  \item The regular tetrahedron formation $\Omega_T$ is locally exponentially stable.
\end{enumerate}
\end{theorem}
\begin{proof}
Through coordinates transformation (\ref{eq:SF:newCoorTrans}), we can investigate the stability of  (\ref{eq:SF:newCorDynf}) instead of that of original system (\ref{eq:PF:ClosedSystem}). Since (\ref{eq:SF:newCorDynf}) shares a triangle form, the stability of subsystem (\ref{eq:SF:newCorDynf1}) can be determined independently. Let $\overline{\xi}_s=\xi_s-\xi_s^*$, and the linearization of this subsystem about $\xi_s^*$ is
\begin{equation}\nonumber
  \begin{split}
    \dot{\overline{\xi}}_s\!=\!\!\left.\left[ \frac{\partial \!\hat{f}_s(\xi_s\!)}{\partial\xi_1},\!\!\frac{\partial \!\hat{f}_s(\xi_s\!)}{\partial\xi_2},\!\!\frac{\partial \!\hat{f}_s(\xi_s\!)}{\partial\xi_3},\!\!\frac{\partial \!\hat{f}_s(\xi_s\!)}{\partial\xi_4},\!\!\frac{\partial \!\hat{f}_s(\xi_s\!)}{\partial\xi_5},\!\!\frac{\partial \!\hat{f}_s(\xi_s\!)}{\partial\xi_6} \right]\right|_{\xi_s^*}\!\!\!\!\cdot \!\overline{\xi}_s\\
   \end{split}.
\end{equation}
The elements of the above matrix satisfy
\begin{equation}\label{eq:SF:proofCM}
  \begin{split}
  \frac{\partial \!\hat{f}_s(\xi_s\!)}{\partial\xi_i}&=\frac{\partial }{\partial\xi_i} \left[h(\xi_s)\widehat{F}(\xi_s)\right] \\
    &=\frac{\partial h(\xi_s\!)}{\partial\xi_i} \widehat{F}(\xi_s)+\dot{F}(\xi_i)h(\xi_s)\mathbf{e}_i .
\end{split}
\end{equation}
where $\mathbf{e}_i$ is the $i$-th column of the identity matrix $\mathbf{I}_{6\times6}$.

Since there exist $F(\cdot),\dot{F}(\cdot)$ in (\ref{eq:SF:proofCM}), the spectrum of the linearized subsystem cannot be obtained directly. However for the formation in $\Omega_T$, all the elements of $\xi_s^*$ are identical, which allows us to separate $F(\cdot),\dot{F}(\cdot)$ as
\begin{gather}\label{eq:SF:proofLinear6}
   \dot{\overline{\xi}}_s=\left(A_\xi^1(\xi_0^*)\cdot F(\xi_0^*)+A_\xi^2(\xi_0^*)\cdot \dot{F}(\xi_0^*)\right)\overline{\xi}_s,
\end{gather}
where $A_\xi^1(\xi_0^*), A_\xi^2(\xi_0^*) \in \mathbb{R}^{6 \times 6}$, $\xi_0^*$ is the element of $\xi_s^*$, equal to $-\frac{1}{3}$.

Furthermore, since the inter-agent graph is completed, $A_\xi^1(\xi_0^*), A_\xi^2(\xi_0^*)$ in (\ref{eq:SF:proofLinear6}) will share the same structure, whose elements except those on diagonals or counterdiagonals are identical and elements on counterdiagonals are 0. Actually this can be also observed in the dynamics of $\xi_s$ (\ref{eq:SF:newCorDynf1Detial}). Hence the matrices $A_\xi^1(\xi_0^*), A_\xi^2(\xi_0^*)$ can be written by the combination of $\Lambda_n$ and $I$:
\[
\begin{split}
    A_\xi^1(\xi_0^*)&=(\xi_0^*-1)*\Lambda_n+(7\xi_0^*+1)I, \\
    A_\xi^2(\xi_0^*)&=(\xi_0^{*2}-\xi_0^*)*\Lambda_n+(\xi_0^{*2}+\xi_0^*-2)I,
  \end{split}\]
where $I$ is the identity matrix, $\Lambda_n$ is the matrix defined as that in  \emph{Lemma} \ref{Thm:PR:lambdaMatrix}, and $n=6$. So the system (\ref{eq:SF:proofLinear6}) can be rewritten as
\begin{equation}
\begin{split}
 \dot{\overline{\xi}}_s&=\Big[(\xi_0^*\!-\!1)\left(F(\xi_0^*)\!+\!\xi_0^*\dot{F}(\xi_0^*)\right)\Lambda_n  \\
        & \;+\!\left((7\xi_0^*\!+\!1)F(\xi_0^*)\!+\!(\xi_0^{*2}\!+\!\xi_0^*\!-\!2 )\dot{F}(\xi_0^*)  \right)I \Big] \overline{\xi}_s,
\end{split}
\end{equation}

By \emph{Lemma} \ref{Thm:PR:lambdaMatrix}, and substituting $\xi_0^*=-\frac{1}{3}$, the spectrum of system (\ref{eq:SF:proofLinear6}) is
\[
\begin{split}
    \left\{ -8F(\xi_0^*),\,  -\frac{8}{3}\dot{F}(\xi_0^*),\, -\frac{8}{3}\dot{F}(\xi_0^*),\, -\frac{8}{3}F(\xi_0^*)\!-\!\frac{16}{9}\dot{F}(\xi_0^*),\right.\\
  \left. -\frac{8}{3}F(\xi_0^*)\!-\!\frac{16}{9}\dot{F}(\xi_0^*),\, -\frac{8}{3}F(\xi_0^*)\!-\!\frac{16}{9}\dot{F}(\xi_0^*) \right\}.
  \end{split}
\]
If $f(\cdot)> 0, \; \dot{f}(\cdot)< 0 $, then $F(\cdot)> 0, \; \dot{F}(\cdot)> 0 $. So for a formation $\mathbf{\Gamma}^*\in \Omega_T$, all the eigenvalues of subsystem (\ref{eq:SF:newCorDynf1}) is negative near $\mathbf{\Gamma}^*$.

By \emph{Lemma} \ref{Thm:SF:ExistCM}, there exist a three-dimensional center manifold in (\ref{eq:SF:newCorDynf}) near the formation $\mathbf{\Gamma}^*\in \Omega_T$. We have already shown that the eigenvectors of (\ref{eq:SF:newCorDynf1}) generates a stable manifold, so the eigenvectors of the rest 3-dimensional subsystem (\ref{eq:SF:newCorDynf2}) will generate the center manifold $M_c(\mathbf{\Gamma}^*)$.

Since the dynamics of $\xi_s$ is independent on $\xi_c$, the center manifold can be determined as
$$M_c(\mathbf{\Gamma}^*)=\left\{ \mathbf{\Gamma} \in (\mathcal{S}^2)^4 : \xi_s=\xi^*_s, [\xi_s^{T},\xi_c^{T}]^T=\Phi(\mathbf{\Gamma}) \right\}.$$
and it is $\Omega_T$ exactly. The second assertion will be obvious, as any center manifold is locally exponentially stable.
\end{proof}

By Remark \ref{Rmk:SF:OmegaTinEqm}, any $\mathbf{\Gamma} \in \Omega_T$ is an equilibrium of system (\ref{eq:SF:newCorDynf}), i.e., the rotational angular velocities of the formation $\mathbf{\Gamma} $ will be zero. Hence this stable regular tetrahedron formation will be stationary.

\subsection{Almost Global Stability of Regular Tetrahedron Formations}

Since $(\mathcal{S}^2)^n$ is a compact space without boundary, under continuous time-invariant feedback control, the global asymptotic stability of the desired formation cannot be achieved. Thus we will devote this subsection to the almost globally asymptotical stability of regular tetrahedron formations.

To this end, we need to prove that except the desired formations $\Omega_T$, other equilibria in set $\Omega$ are all unstable. We need to make a further restriction on the control gain.

\begin{assumption}\label{As:SFG:f}
  The gain function $f(\cdot)$ in (\ref{eq:PF:Control}) has the structure $f(\theta_{ij})=a\cos(\theta_{ij})+a$, where $a \in \mathbb{R}$ and $a>0$.
\end{assumption}

We note that the form of function $f(\cdot)$ in Assumption \ref{As:SFG:f} also fulfills Assumption \ref{As:SF:f}. The next lemma provides an important inequality which will be exploited in the rest of this section.

\begin{lemma}\label{Thm:SFG:Inequality}
For any four angles $\theta_i \in [0,\pi]$, $i \in \mathcal{V}= \{1,2,3,4\}$, if $\theta_1+\theta_2+\theta_3+\theta_4=2\pi$, then we have
$$ \sum\limits_{i\in \mathcal{V}}{\big[\cos^2(\theta_i)+\cos(\theta_i)}\big] \ge 0 ,$$
in which the equality holds if and only if $\theta_i = \pi/2, \forall i \in \mathcal{V}$.
\end{lemma}
This lemma is obtained by minimizing the left-hand scalar function under the above constraints. A detailed proof can be found in the appendix.

\begin{lemma}\label{Thm:SFG:CoplaneUnStable}
Under Graph \ref{As:SF:completegraph} and Assumption \ref{As:SFG:f}, any equilibrium in $\Omega_C$ of the closed-loop system (\ref{eq:PF:ClosedSystem}) is unstable,  where
\begin{equation*}
\begin{split}
  \Omega_C = \Big\{ \mathbf{\Gamma} \in (\mathcal{S}^2)^4 : \Gamma_i= exp(\theta_i \widehat{u})\Gamma_1, \; \Gamma_1 \in \mathcal{S}^2, u\in \mathcal{S}^2,  \\
 \theta_i \in [-\pi,\pi), i \in \{2,3,4\}  \Big\},
\end{split}
\end{equation*}
 i.e. any formation consisting of four coplanar reduced attitudes is unstable.
\end{lemma}

\begin{proof}
  Denote $e=[0,0,1]^T \in  \mathcal{S}^2$ as the north pole. For any equilibrium $\mathbf{\Gamma}^* \in \Omega_C$, if  $u$ corresponding to  $\mathbf{\Gamma}^*$ is not equal to $\pm e$ , let $\alpha = \arccos(u^Te)$ and $v=\widehat{u}e / \sin(\alpha)$. Through the coordinate transformation $ \overline{\Gamma}_i= exp(\alpha \widehat{v})\Gamma_i, i \in \mathcal{V}$, $\mathbf{\Gamma}^*$ can be transformed to an equilibrium lying on the equator. By Lemma \ref{Thm:PF:rotationalInvariant}, $\mathbf{\Gamma}^*$ and $\overline{\mathbf{\Gamma}}^*$ share the same stability. Hence, without loss of generality, we suppose $u=e$. Thus we have $\phi_i^*=0$ and $\cos(\theta_{ij}^*)=\cos(\psi_i^*-\psi_j^*)$, $i,j \in \mathcal{V}$.

Let $\overline{\psi}_i=\psi_i-\psi_i^*$, and linearize the system (\ref{eq:PF:ClosedSysRPY1})-(\ref{eq:PF:ClosedSysRPY2}) around $\{(\psi_i^*,\phi_i^*)\}_{i \in \mathcal{V}}$. We get

\begin{subequations}
    \begin{equation}\label{eq:SFG:ClosedSysRPY1}
      \dot{\overline{\psi}}_i=\sum\limits_{j\in \mathcal{N}_i}{\left[f(\theta_{ij}^*)\cos(\theta_{ij}^*)+\dot{f}(\theta_{ij}^*)\sin(\theta_{ij}^*)\right] (\overline{\psi}_i-\overline{\psi}_j)} ,
    \end{equation}
    \begin{equation}\label{eq:SFG:ClosedSysRPY2}
        \dot{\phi}_i =\sum\limits_{j\in \mathcal{N}_i}{f(\theta_{ij}^*)\left[ cos(\theta_{ij}^*)\phi_i-\phi_j \right]}.
    \end{equation}
\end{subequations}
The above linearizations are decoupled, and we denote the system matrices of (\ref{eq:SFG:ClosedSysRPY2}) as $A_{\psi}(\mathbf{\Gamma}^*  )  \in \mathbb{R}^{4 \times 4}$. We note that $A_{\psi}(\mathbf{\Gamma}^*  )$ is a symmetric matrix. Now let a vector $v=[1,0,-1,0]^T$, then
\begin{equation}\label{eq:SFG:vAv}
\begin{split}
  v^TA_{\psi}(\mathbf{\Gamma}^*)v= \!\sum\limits_{j=i+1,i\in \mathcal{V}}\! {\left[ a \cos^{2}(\theta_{ij})\!+\!a\cos(\theta_{ij})\right] } \\
   + 2a(\cos(\theta_{13})+1)^2,
\end{split}
\end{equation}
where a modulo $n$ operation is applied for agent identification. Suppose $\theta_{i,i+1} \le \pi$, $\forall i \in \mathcal{V}$, otherwise a contradiction can be obtained easily from the fact that ${\Gamma}^*$ is one equilibria. By Lemma \ref{Thm:SFG:Inequality}, $v^TA_{\psi}(\mathbf{\Gamma}^*)v \ge 0$, and $v^TA_{\psi}(\mathbf{\Gamma}^*)v = 0$ when $\mathbf{\Gamma}^* \in \Omega_{\overline{C}}=\left\{ \mathbf{\Gamma} \in (\mathcal{S}^2)^4 : \Gamma_i^T\Gamma_j \in \{0, -1\}, \forall i \neq j \right\}$, i.e. $\mathbf{\Gamma}^*$ constructs a cross formation.
Hence, for any $\mathbf{\Gamma}^* \in \Omega_C - \Omega_{\overline{C}}$, at least one of eigenvalue of $A_{\psi}(\mathbf{\Gamma}^*) $ is positive.

For any $\mathbf{\Gamma}^* \in \Omega_{\overline{C}}$, without loss of generality, we assume $\psi_i^*=-\pi+i\pi/2$, $i \in \mathcal{V}$. Then the eigenvalues of $A_{\psi}(\mathbf{\Gamma}^*) $ are $\{ \pm2a, 0,0\}$. Combining with the above proof, the assertion follows.
\end{proof}

Next we investigate the components of set $\Omega$, by which we can show $\Omega_T$ is the only stable manifold in $\Omega$.

\begin{lemma}\label{Thm:SF:Equilibria}
Under Graph \ref{As:SF:completegraph} and Assumption \ref{As:SFG:f}, the equilibria set of  closed-loop system (\ref{eq:PF:ClosedSystem}) can be decomposed as
\begin{gather}\label{eq:SF:Equilibria}
  \Omega=\Omega_T \cup \Omega_L,
\end{gather}
where $\Omega_L \subseteq \Omega_C$.
\end{lemma}

\begin{proof}
We consider two cases separately, according to whether there exist $\Gamma_i^T\widehat{\Gamma}_j\Gamma_k=0$ for some mutually unequal $i$,$j$,$k$.

If $\Gamma_i^T\widehat{\Gamma}_j\Gamma_k\neq0$, then any ${k}_{ij}$ and ${k}_{jk}$ are linearly independent. For any $\mathbf{\Gamma} \in \Omega$, we have $\Gamma_i \parallelsum \sum\limits_{j\in \mathcal{N}_i} f(\theta_{ij})\Gamma_j$, namely,
     \begin{gather}\label{eq:SF:proofequilibria}
       l_i\Gamma_i+\sum\limits_{j\in \mathcal{N}_i} f(\theta_{ij})\Gamma_j=0, \; i\in \mathcal{V},
     \end{gather}
     where $l_i \in \mathbb{R}$, $l_i\neq 0$.

    We denote $f_{ij}:=f(\theta_{ij})$, $i,j\in \mathcal{V}$ and $i\neq j$. By eliminating $\Gamma_4$ from (\ref{eq:SF:proofequilibria}), three independent equations can be obtained. When each one of them is cross-multiplied  with $\Gamma_1$ on the left side, we get
    \begin{equation}\label{eq:SF:proofequilibria2}
       \begin{gathered}
           (f_{12}f_{24}\!\!-\!\!l_2f_{14})\sin(\theta_{12}\!)k_{12}\!\!+\!\! (f_{13}f_{24}\!\!-\!\!f_{23}f_{14})\sin(\theta_{13}\!)k_{13} \!\!=\!\!0\\
            (f_{12}f_{34}\!\!-\!\!f_{23}f_{14})\sin(\theta_{12}\!)k_{12} \!\!+\!\! (f_{13}f_{34}\!\!-\!\!l_3f_{14})\sin(\theta_{13}\!)k_{13} \!\!=\!\!0\\
           (l_2f_{34}\!\!-\!\!f_{23}f_{24})\sin(\theta_{12}\!)k_{12} \!\!+\!\!(f_{23}f_{34}\!\!-\!\!l_3f_{24})\sin(\theta_{13}\!)k_{13}\!\!=\!\!0
       \end{gathered}
     \end{equation}
     Since $k_{12}$ and $k_{13}$ are independent, and $\sin(\theta_{ij})\neq 0$, we obtain
\begin{subequations}
     \begin{equation}\label{eq:SF:proofequilibria3a}
      f_{12}f_{34}=f_{13}f_{24},
    \end{equation}
    \begin{equation}\label{eq:SF:proofequilibria3b}
     f_{12}f_{34}=f_{14}f_{23},
    \end{equation}
    \begin{equation}\label{eq:SF:proofequilibria3c}
     f_{12}f_{13}\Gamma_1+f_{12}f_{23}\Gamma_2+f_{13}f_{23}\Gamma_3+f_{14}f_{23}\Gamma_4=0,
    \end{equation}
\end{subequations}
We dot-multiply equation (\ref{eq:SF:proofequilibria3c}) with $\Gamma_i$, $i \in \mathcal{V}$, respectively and notice that $\Gamma_i^T\Gamma_j=\frac{1}{a}f_{ij}-1$. By the above steps, four equations in terms of $\{f_{ij}\}_{i,j \in \mathcal{V}}$ can be derived from (\ref{eq:SF:proofequilibria3c}). Combining with (\ref{eq:SF:proofequilibria3a}) (\ref{eq:SF:proofequilibria3b}), we construct a polynomial system with respect to $f_{ij}$, $i,j \in \mathcal{V}$. We notice that these equations are symmetric with respect to the index. This polynomial system only has one positive real solution with $f_{ij}=f_{pq}=2a/3$, where $i \neq j, p \neq q$. Since $f(\cdot)$ is monotone, this solution just corresponds to regular tetrahedron formation $\Omega_T$.

If $\Gamma_i^T\widehat{\Gamma}_j\Gamma_k=0$ for $i,j,k \in \mathcal{V}$, then ${k}_{ij} \parallelsum {k}_{jk}$ and $\Gamma_i, \Gamma_j, \Gamma_k$ are coplanar. So we have ${k}_{ij}^T\Gamma_i=0$, ${k}_{ij}^T\Gamma_j=0$, ${k}_{ij}^T\Gamma_k=0$. Assume another reduced attitude is $\Gamma_m$, then by dot multiplying (\ref{eq:SF:proofequilibria}) with ${k}_{ij}$, we have ${k}_{ij}^T\Gamma_m=0$. So $\Gamma_m$ is also on the same plane, namely, the equilibria in this case are the subset of $\Omega_C$.

\end{proof}

Now, we are ready to state the following result.
\begin{theorem}
Under Graph \ref{As:SF:completegraph} and Assumption \ref{As:SFG:f}, a regular tetrahedron formation $\Omega_T $ is almost globally asymptotically stable.
\end{theorem}
\begin{proof}
By Theorem \ref{Thm:SF:stableEquiSet}, Lemma \ref{Thm:SF:Equilibria} and Lemma \ref{Thm:SFG:CoplaneUnStable}, the assertion can be obviously obtained.
\end{proof}

\section{Rotating  regular Tetrahedron Formation}

In this section, we will construct a regular tetrahedron formation which spins about an axis passing through one of the agents at a constant speed. To this aim, we will redesign the inter-agent connection as a weighed graph. First we will equip a graph  $\mathcal{G}$ with weight.

Here, we define a weighed graph $\mathcal{\widetilde{G}}$ as a pair $(\mathcal{G}, \mathcal{W}(e))$, where $\mathcal{G}=(\mathcal{V},\mathcal{E})$ is a graph, and  $\mathcal{W}:\mathcal{E} \rightarrow \mathbb{R}$ is a weight function assigned to each edge $e \in \mathcal{E}$. In addition, for the nodes in $\mathcal{V}=\{1,\ldots,n\}$ , we introduce the following two operators. The node $i$'s \textit{successor skipping} $k$ is denoted by $Suc^k(i)$, which is defined as
$$Suc^k(i) =\begin{cases}
      (i+1)\, mod\, n, & if\; i+1 \neq k\\
      (i+2)\, mod\, n, & if\; i+1 = k
    \end{cases},
$$
where $mod$ is the modulo operator. And $Pre^k(i)$ is the node except $k$ whose successor skipping $k$ is $i$, and it is referred to as node $i$'s \textit{predecessor skipping} $k$.

Now we redesign the inter-agent connection as a weighed graph to obtain a rotating regular tetrahedron.
\begin{graph}\label{As:RF:weighedgraph}
  The inter-agent weighed graph $\mathcal{\widetilde{G}}=(\mathcal{G}, \mathcal{W}(e))$, where $\mathcal{G}$ is a completed graph with 4 nodes, and $\mathcal{W}\left((i,j)\right)$ satisfying
$$\mathcal{W}\left((i,j)\right) =\begin{cases}
      0.5, & if\;  i \neq j, k \in \{i,j\}\\
      p,   & if \; k \notin \{i,j\}, j=Suc^k(i)\\
      1-p, & if \; k \notin \{i,j\}, j=Pre^k(i)
    \end{cases},
$$
where $(i,j) \in \mathcal{E}$, $k \in \mathcal{V}$ determines the rotation axis, and $p \in \{0,1\}$ specifies the direction of rotation. For simplicity, when there is no ambiguity, we use $\mathcal{W}\left(i,j\right)$ instead of $\mathcal{W}\left((i,j)\right)$.
\end{graph}
Two examples of Graph \ref{As:RF:weighedgraph} with different $k$ and $p$ are depicted in Fig. \ref{Fig:RF:WGraph}.
\begin{figure}
  \centering
  \includegraphics[width=0.45\textwidth]{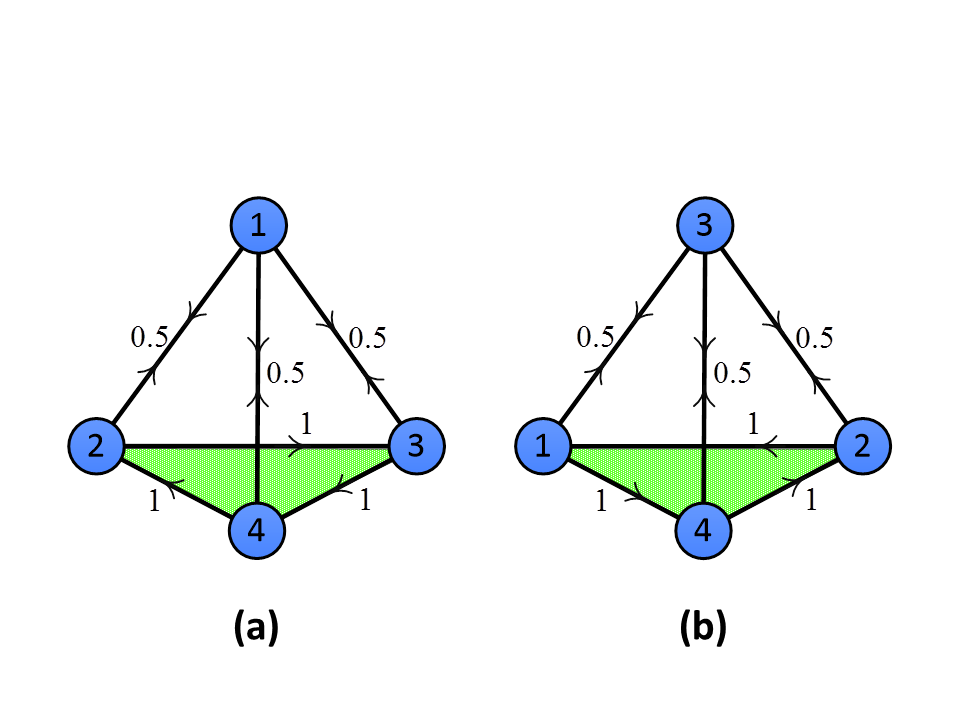}\\
  \caption{Depiction of inter-agent connection for different $k$ and $p$.  The graph in (a) is in the case $k=1$, $p=1$. And that in (b) $k=3$, $p=0$.}\label{Fig:RF:WGraph}
\end{figure}

\begin{remark}
  Graph \ref{As:RF:weighedgraph} is a connection relationship where there is one agent with bidirectional access to any other agents, and the rest of agents construct a directed cycle graph.
\end{remark}
By reordering the agents, without loss of generality, we can assume $k=1$ and $p=1$ throughout the rest of this section. Meanwhile the control law need a slight improvement to include the weight of Graph \ref{As:RF:weighedgraph}.

\begin{equation}
    \widetilde{\omega}_i=-\sum_{j\in \mathcal{N}_i} \mathcal{W}(i,j)f(\theta_{ij})\widehat{\Gamma}_i\Gamma_j, \;\; \forall i\in \mathcal{V}, \label{eq:RF:Control}
\end{equation}
and the closed-loop system becomes
\begin{equation}
   \dot{\Gamma}_i=\widehat{\Gamma}_i \sum_{j\in \mathcal{N}_i} \mathcal{W}(i,j)f(\theta_{ij})\widehat{\Gamma}_i\Gamma_j, \;\; \forall i\in \mathcal{V}. \label{eq:RF:ClosedSystem}
\end{equation}

Under the coordinates transformation of (\ref{eq:SF:newCoorTrans}), the closed-loop system (\ref{eq:RF:ClosedSystem}) can be transformed into a triangle system
    \begin{subequations}
        \begin{equation}\label{eq:RF:newCorDynf1}
          \dot{\xi}_s= \widetilde{f}_s(\xi_s),
        \end{equation}
        \begin{equation}\label{eq:RF:newCorDynf2}
            \dot{\xi}_c= \widetilde{f}_c(\xi_c,\xi_s),
        \end{equation}
    \end{subequations}
where the explicit expression of $\widetilde{f}_s(\cdot)$ can be obtained, which shares a similar form with (\ref{eq:SF:newCorDynf1Detial}). For simplicity, we omit the explicit expression here.

\begin{theorem}\label{thm:RF:LocalStable}
  Under Graph \ref{As:RF:weighedgraph} and Assumption \ref{As:SF:f}, in closed-loop system (\ref{eq:RF:ClosedSystem}), the regular tetrahedron formations $\Omega_T$ is an invariant manifold, and it is also exponentially stable.
\end{theorem}
\begin{proof}
We can verify that $\widetilde{f}_s(\xi^*_s)=0$, where $\xi^*_s$ satisfies (\ref{eq:SF:equiNewCoor}). So $\Omega_T=\left\{(\xi_s^T,\xi_c^T)^T:\xi_s= \xi_s^* \right\}$ is an invariant manifold of system (\ref{eq:RF:newCorDynf1})-(\ref{eq:RF:newCorDynf2}). By applying the same approach in the proof of Theorem \ref{Thm:SF:CMDetile}, we can linearize the closed-loop system (\ref{eq:RF:newCorDynf1}) around $\xi^*_s$, and the system matrix of this linearized system is Hurwitz. Thus, $\Omega_T$ is exponentially stable.
\end{proof}

Theorem \ref{thm:RF:LocalStable} gives local stability of regular tetrahedron formations $\Omega_T$, but in order to reveal the motions of the whole formation already formed, we need to investigate the dynamics of subsystem (\ref{eq:RF:newCorDynf2}) on the manifold $\Omega_T$.

First, for every $\mathbf{\Gamma}\in (\mathcal{S}^2)^4$, we define a normalized attitude $\widetilde{\mathbf{\Gamma}} \in (\mathcal{S}^2)^4$, such that $\widetilde{\mathbf{\Gamma}}$ and $\mathbf{\Gamma}$ has the same formation, i.e. $\widetilde{\Gamma}_i^T\widetilde{\Gamma}_j=\Gamma_i^T\Gamma_j$, for all $i,j \in \mathcal{V}$, and $\widetilde{\Gamma}_1$ is overlapping with the north pole, and $\widetilde{\Gamma}_2$ locates on the plane $\emph{XOZ}$. From the rotational process introduced by the transformation (\ref{eq:SF:newCoorTrans}), we have $\Gamma_i= R_3(\psi_1)R_2(\phi_1)R_1(\gamma)\widetilde{\Gamma}_i$, where $i\in \mathcal{V}$ and $\left\{R_j(\cdot)\right\}_{j=1,2,3}$ is defined in (\ref{eq:SF:rotMatrix}). Unless necessary, we omit the explicit dependence of $R_j$ on its argument. Then we take the derivative of this relationship with respect to time $t$,
\begin{equation}\label{eq:RF:SpeedDerive1}
\begin{split}
  \dot{\Gamma}_i=R_3R_2R_1 & \sum\limits_{l=1}^6 {\frac{\partial \widetilde{\Gamma}_i}{\partial \xi_l} \dot{\xi}_l}  +R_3R_2\frac{\partial R_1}{\partial \gamma}\dot{\gamma}\\
  &\;+R_3\frac{\partial R_1}{\partial \phi_1}R_1\dot{\phi_1}+\frac{\partial R_3}{\partial \psi_1}R_2R_1\dot{\psi_1}.
\end{split}
\end{equation}
On the other hand, from the dynamics (\ref{eq:RF:ClosedSystem}), and the fact $\widehat{\Gamma}_i\widehat{\Gamma}_i\Gamma_j=\cos(\theta_{ij})\Gamma_i-\Gamma_j$, we have
\begin{equation}\label{eq:RF:SpeedDerive2}
   \dot{\Gamma}_i= R_3R_2R_1\! \sum_{j\in \mathcal{N}_i} \mathcal{W}(i,j)f(\theta_{ij})\left[ \cos(\theta_{ij})\widetilde{\Gamma}_i\!-\!\widetilde{\Gamma}_j \right]
\end{equation}

As we consider the motion of system restricted in the invariant manifold $\Omega_T$, we have $\theta_{ij}=\theta^*=\arccos(-\frac{1}{3})$, $i\neq j \in \mathcal{V}$, and $\dot{\xi}_l=0$, for $l \in \{1,2,\cdots,6\}$, so that $\widetilde{\Gamma}_1=[0, 0 ,1]^T$, $\widetilde{\Gamma}_2=[-\frac{2\sqrt{2}}{3} ,0 ,-\frac{1}{3}]^T$, $\widetilde{\Gamma}_3=[-\frac{\sqrt{2}}{3} ,\frac{\sqrt{6}}{3} , -\frac{1}{3}]^T$, $\widetilde{\Gamma}_4=[-\frac{\sqrt{2}}{3} ,-\frac{\sqrt{6}}{3} ,-\frac{1}{3}]^T$. By plugging (\ref{eq:RF:SpeedDerive2}) into (\ref{eq:RF:SpeedDerive1}), we get twelve scalar equations, but only three of them are independent, which are
\begin{equation*}
\begin{split}
\dot{\psi}_1\cos(\phi_1)\sin(\gamma)-\dot{\phi}_1\cos(\gamma)=0,\\
\dot{\phi}_1\sin(\gamma)+\dot{\psi}_1\cos(\gamma)\cos(\phi_1)=0,\\
\sqrt{3}f(\theta^*)+2\dot{\gamma}+2\dot{\psi}_1\sin(\phi_1)=0.
\end{split}
\end{equation*}
By solving these equations, we have $\dot{\phi}_1=0$, $\dot{\psi}_1=0$, and $\dot{\gamma}=-\frac{\sqrt{3}}{2}f(\theta^*)$.  This result can be interpreted as follows.

\begin{proposition}
In the regular tetrahedron formation constructed in Theorem \ref{thm:RF:LocalStable}, the reduce attitude $\Gamma_1$ remains stationary, and the whole formation spins about a axis passing through $\Gamma_1$ with a constant angular velocity $\omega=\frac{\sqrt{3}}{2}f(\theta^*)$.
\end{proposition}

\section{Simulation}\label{sec:simu}
In this section, we present several numerical examples to show convergence of the desired formation governed by the proposed control.

First, we start with the stationary tetrahedron formation, in which the inter-agent graph is complete.
Two specific gain functions satisfying Assumption \ref{As:SF:f} are simulated successively, which are $f_1(\theta)=\cos(\theta)+1$ and $f_2(\theta)=e^{-\theta}$. The simulation result is presented in Fig.~\ref{Fig:Sim:SF}. From (c) and (d) in Fig.~\ref{Fig:Sim:SF}, the geodesic distance between every two reduced attitudes converges to $\arccos(-\frac{1}{3})$ eventually, namely, the formation of four reduced attitudes converge to a stationary regular tetrahedron in $\mathcal{S}^2$.

\begin{figure}[t]
    \centering
    \subfigure[Movement trajectory of reduced attitudes system with $f_1$.]{%
        \includegraphics[width=0.21\textwidth]{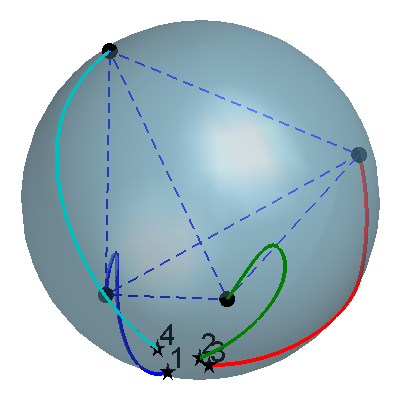}
        \label{Fig:Sim:SF_Cos}} \,
    \subfigure[Movement trajectory of reduced attitudes system with $f_2$.]{%
        \includegraphics[width=0.21\textwidth]{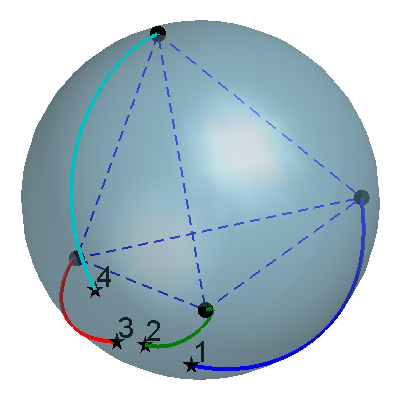}
        \label{Fig:Sim:SF_Exp}}
    \subfigure[Cosines of geodesic distance between every two reduced attitudes with $f_1$.]{%
        \includegraphics[width=0.45\textwidth]{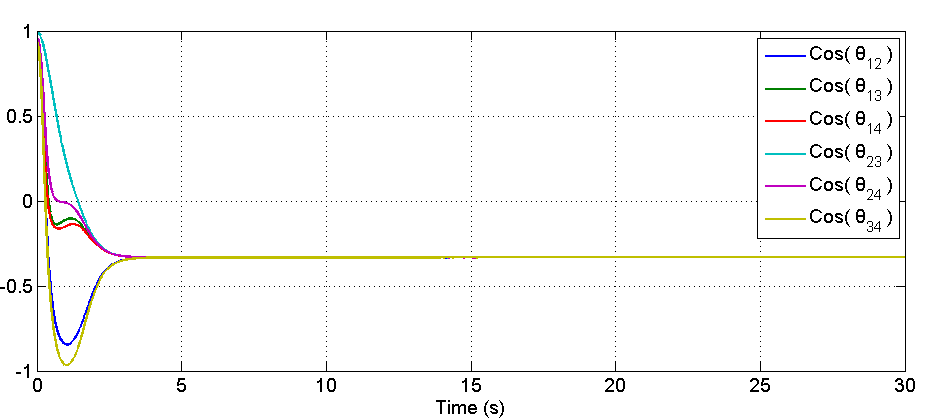}
        \label{Fig:Sim:SF_Cos1}}
    \subfigure[Cosines of geodesic distance between every two reduced attitudes with $f_2$.]{%
        \includegraphics[width=0.45\textwidth]{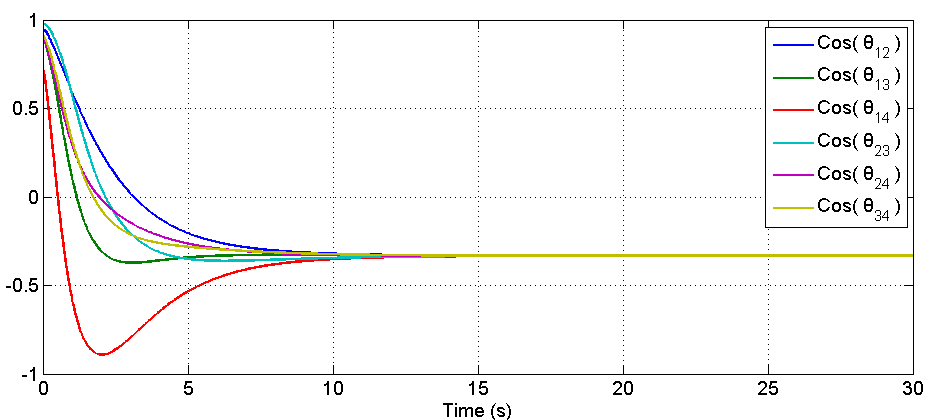}
        \label{Fig:Sim:SF_Exp1}}
    \caption{Simulation result of two reduced attitude systems with different gain functions $f_1(\theta)=e^{-\theta}$ and $f_2(\theta)=\cos(\theta)+1$. }
    \label{Fig:Sim:SF}
\end{figure}

Then, we turn the inter-agent graph to the weighed Graph \ref{As:RF:weighedgraph} with $k=1$, $p=1$. Two different gain functions are applied, which are $f_1(\theta)=\cos(\theta)+1$ and $f_2(\theta)=e^{-\theta}$. In both cases, the reduced attitudes construct a rotating tetrahedron formation, as presented in Fig.~\ref{Fig:Sim:RF}. From (c) and (d) in Fig.~\ref{Fig:Sim:RF} , we can see that in the formation formed, $\Gamma_1$ is stationary and the other three attitudes spin with a constant angular velocity. The period of this rotation is exactly $2\sqrt{3}\pi$.

\begin{figure}[t]
    \centering
    \subfigure[Movement trajectory of reduced attitudes system with $f_1$.]{%
        \includegraphics[width=0.21\textwidth]{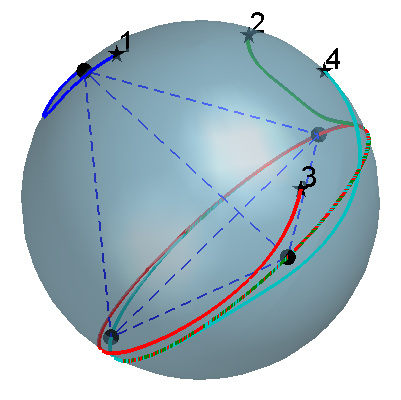}
        \label{Fig:Sim:RF_Cos}}
    \,
    \subfigure[Movement trajectory of reduced attitudes system with $f_2$.]{%
        \includegraphics[width=0.21\textwidth]{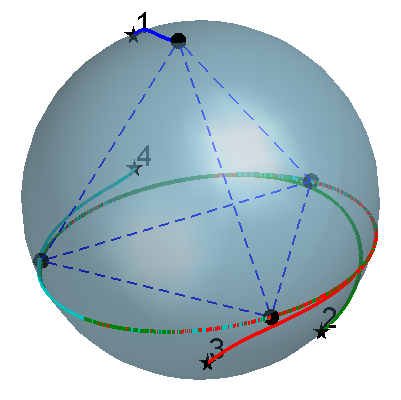}
        \label{Fig:Sim:RF_Exp}}
    \subfigure[Three components of every reduced attitude vector with $f_1$.]{%
        \includegraphics[width=0.45\textwidth]{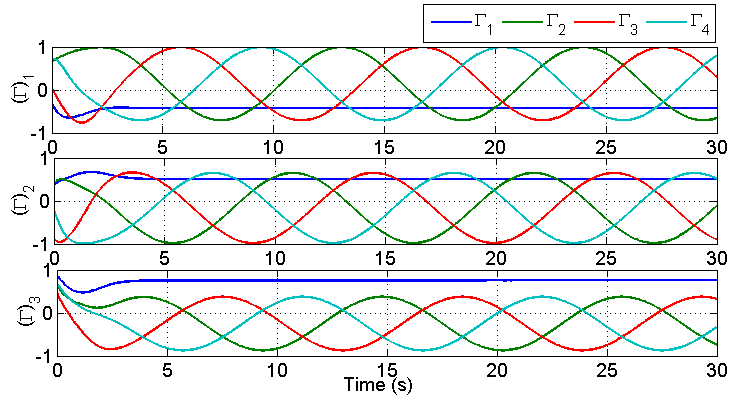}
        \label{Fig:Sim:RF_Cos1}}
    \caption{Simulation result of two reduced attitude systems with different gain functions $f_1(\theta)=e^{-\theta}$ and $f_2(\theta)=\cos(\theta)+1$. }
    \label{Fig:Sim:RF}
\end{figure}

\section{Concluding Remark}\label{sec:conclusion}
In this paper, we presented an intrinsic formation control method for reduced attitudes, by which a specific solid formation, regular tetrahedron, was rendered almost globally stable, and a rotating formation become locally exponentially stable. The proposed control did not contain any information of the desired formation beforehand, and only depended on the relative reduced attitudes between the agents. In future work, various formation patterns will be investigated, by which the intrinsic formation results on attitude space will be gradually extended.


\bibliographystyle{plain}      
\bibliography{ref}           



\appendix
\section{Proofs}
\subsection{Proof of \textit{Lemma} \ref{Thm:PR:lambdaMatrix}}    

We note that $\Lambda_n=E_n-J$, where $E_n$ is the matrix whose elements are all $1$'s. Since $E_n$ is both symmetric and persymmetric, \emph{Lemma} \ref{Thm:PR:AJCommutative} gives that $E_n$ and $J$ are commutative. By \emph{Lemma} \ref{Thm:pre:simuDiagonalize}, $E_n$ and $J$ can be simultaneously diagonalized. Thus, the eigenvalues of $\Lambda_n$ is a combination of those of $E_n$ and $J$, namely
$$\lambda_i(\Lambda_n)=\lambda_{p(i)}(E_n)+\lambda_{q(i)}(J),$$
where $\lambda_i(\cdot)$ is the $i$-th eigenvalue and $\left\{p(i)\right\}_{i=1}^6$, $\left\{q(i)\right\}_{i=1}^6$ are two specific permutations of $\{1,2,\cdots,6\}$. Meanwhile, when we add all the other rows in $det(\lambda I-E_n)$ to the last row and eliminate the rows except the last one, we have $det(\lambda I-E_n)=(\lambda-n)\lambda^{n-1}$. Hence $\lambda_{i}(E_n)=\{n,0,\cdots,0\}$. The same operation will give us $\lambda_{i}(J)=\{(-1)^{i}\}_{i=1}^6$, and $det(\lambda I-\Lambda_n)$ must contain a factor $(\lambda-n+1)$. Then the assertion follows.

\subsection{Proof of \textit{Lemma} \ref{Thm:SFG:Inequality}}         

We rewrite the problem as the following optimization problem
\begin{equation*}
\begin{split}
  \min_{X} &\;G(X)=\sum\limits_{i \in \mathcal{V}} [\cos^2(x_i)+\cos(x_i)]\\
  s.t.\;& A_IX \ge b_I\\
  &A_EX=2\pi
\end{split}
\end{equation*}
where $X=[x_1,x_2,x_3,x_4]^T$, $A_E=[1,1,1,1]^T$, $b_I=[-\pi,-\pi,-\pi,-\pi,0,0,0,0]^T$, $A_I=[-I_4,I_4]^T$ and we denote $I_4 \in \mathbb{R}^{4 \times 4}$ the identity matrix. Then by KKT condition, introducing Lagrange-multiplier $\mathbf{\lambda}=[\lambda_I^T,\lambda_E^T ]^T$ where$ \lambda_I \in \mathbb{R}^8, \lambda_E \in \mathbb{R} $, we have the necessary condition for the above optimization problem as
\begin{equation*}
  (KKT)\begin{cases}
  \nabla G(X^*)=[A_I^T\; A_E^T]\mathbf{\lambda}^*\\
  \lambda_I^*(A_IX-b_I) =0\\
  A_EX^*=2\pi\\
  \lambda_I^* \ge 0\\
  A_IX^* \ge b_I.
\end{cases}
\end{equation*}
By solving this condition, we can get three groups of solutions. If we ignore the order among $x_i$, these three solutions are $(0,2/3\pi,2/3\pi,2/3\pi)$, $(0,0,\pi,\pi)$, $(\pi/2,\pi/2,\pi/2,\pi/2)$. The minimal objective value comes from $X^*=(\pi/2,\pi/2,\pi/2,\pi/2)^T$, and $G(X^*)=0$.

\end{document}